\documentclass[letterpaper,11pt,reqno]{amsart} 
\usepackage[portrait,margin=1in]{geometry} 

\usepackage{mathrsfs,xfrac} 
\usepackage[colorlinks=true,linkcolor=blue,citecolor=blue,urlcolor=blue]{hyperref} 
\usepackage{amsmath,amssymb,amsthm,amsfonts,amsbsy,latexsym,dsfont,color,graphicx,enumitem}
\usepackage[foot]{amsaddr}
\usepackage{caption}
\usepackage{regexpatch}
\usepackage{todonotes}
\makeatletter
\xpatchcmd{\@todo}{\setkeys{todonotes}{#1}}{\setkeys{todonotes}{inline,#1}}{}{}
\makeatother

\newtheorem{thm}{Theorem}[section]
\newtheorem{lem}[thm]{Lemma}
\newtheorem{cor}[thm]{Corollary}
\newtheorem{prop}[thm]{Proposition}

\newtheorem{remark}{Remark}

\renewcommand{\le}{\leqslant}  
\renewcommand{\ge}{\geqslant} 
\renewcommand{\leq}{\leqslant}  
\renewcommand{\geq}{\geqslant} 

\newcommand{\ra}{\rangle}
\newcommand{\la}{\langle}

\newcommand{\eps}{\varepsilon}

\newcommand{\norm}[1]{\left\Vert#1\right\Vert}
\newcommand{\abs}[1]{\left\vert#1\right\vert}

\newcommand{\equald}{\stackrel{\mathrm{d}}{=}}

\def\qed{ \hfill $\blacksquare$}  
              \let\gs=\sigma

         \let\gS=\Sigma  
                                
\newcommand{\cO}{\mathcal{O}}

\newcommand{\mvgr}{\boldsymbol{\rho}}
\newcommand{\mvgs}{\boldsymbol{\sigma}}\newcommand{\mvgt}{\boldsymbol{\tau}}

\newcommand{\dR}{\mathds{R}}

\DeclareMathOperator{\E}{\mathds{E}}
\DeclareMathOperator{\pr}{\mathds{P}}
\DeclareMathOperator{\N}{N}

\DeclareMathOperator{\var}{Var}

\DeclareMathOperator{\dtv}{d_{TV}}
\begin{document}
\title[SK model near the critical temperature]{Fluctuations for the Sherrington--Kirkpatrick spin glass model near the critical temperature}
\author[Dey]{Partha S.~Dey$^\star$}
\author[Kang]{Taegu Kang$^\dag$}
\date{\today}
\address{Department of Mathematics, University of Illinois Urbana--Champaign, 1305 W.~Green Street, Urbana, Illinois 61801}
\email{$^\star$psdey@illinois.edu, $^\dag$taeguk2@illinois.edu}

\begin{abstract}
        We consider the Sherrington--Kirkpatrick spin glass model with zero external field and at inverse temperature $\beta>0$. Let $F_N(\beta)$ be the corresponding log-partition function. Under the assumption that $c_N:=N^{1/3}(1-\beta_N^2)$ is bounded away from $0$, we prove that 
        $
        \var(F_N(\beta_N)) = - \frac{1}{2} \ln (1-\beta_N^2) -{\beta_N^2}/{2} + \cO( c_N^{-3/2}).
        $
        As a consequence, we obtain $\var(F_N(1-c N^{-1/3})) = \frac16\ln N + O(1)$ for any fixed constant $c\in(0,\infty)$.
        We also prove a Gaussian central limit theorem for the centered and scaled $F_N(\beta_N)$.
\end{abstract}

\maketitle

%%%%%%%%%%%%%%%%%%%%%%%%%%%%%%%
\section{Introduction}

We consider the Sherrington-Kirkpatrick (SK) spin glass model at inverse temperature \(\beta>0\) and with zero external field.
For each positive integer $N$, we define the configuration space \( \Sigma_N = \{ -1, +1 \}^N \).
For a configuration \(\mvgs=(\gs_i)_{i\in[N]} \in \Sigma_N\), the Hamiltonian of the SK model is given by
\begin{align*}
    H_{N,\beta}(\mvgs) 
    = \frac{\beta}{\sqrt{N}} \sum_{1 \leq i< j \leq N} g_{ij} \gs_i \gs_j,
\end{align*}
where \(g_{ij}\) are independent and identically distributed (i.i.d.) standard Gaussian random variables.

We define the partition function, Gibbs measure, and free energy for the SK model, respectively, by
\(
Z_N(\beta) = \sum_{\mvgs \in \Sigma_N} \exp({H_{N,\beta} (\mvgs)})\), 
\(G_{N,\beta}(\mvgs) = \exp(H_{N,\beta}(\mvgs)) / Z_N(\beta)\), and 
\( F_N(\beta) = \ln Z_N(\beta)\).

The Sherrington-Kirkpatrick spin glass model has been a subject of intense study since its introduction, with the primary goal of characterizing its limiting free energy. A landmark achievement in this field was the rigorous proof of the Parisi formula, which yields the free-energy limit at any temperature. This was first established by Talagrand~\cite{Tal06} for the standard SK model and later extended to mixed $p$-spin models by Panchenko~\cite{Pan14}.

While the Parisi formula describes the first-order limit of the free energy, the study of its fluctuations focuses on the second-order behavior. In the high-temperature regime $\beta < 1$, the first rigorous proof of the central limit theorem (CLT) for the free energy was established by Aizenman, Lebowitz, and Ruelle~\cite{ALR87}, who showed that for the entire regime $\beta < 1$, the fluctuations of the free energy are Gaussian with $O(1)$ variance. This foundational result was later re-established by Comets and Neveu~\cite{CN95} using an alternative martingale-based approach. In contrast, the low-temperature regime $\beta > 1$ exhibits much larger fluctuations. Chatterjee~\cite{Cha09} provided a general upper bound in this region, proving that the variance of the free energy is at most of the order $N/\ln N$.

As the temperature approaches the critical value $\beta_c = 1$, the fluctuations of the free energy are predicted to diverge. In the physics literature, it is widely believed that at $\beta = 1$, the variance behaves as $\frac{1}{6} \ln N +O(1)$. However, establishing this scaling rigorously remains a formidable challenge. A significant step forward was made by Chen and Lam~\cite{CL19}, who established that the variance at the critical temperature is at most of the order $(\ln N)^2$.

For the \textit{spherical} SK model, which is often more tractable due to its connections with random matrix theory, Baik and Lee~\cite{BL16} provided a complete characterization of the free energy fluctuations, identifying a transition to the Tracy--Widom distribution at low temperatures. Recently, Landon~\cite{Lan22} characterized the near-critical fluctuations for the spherical model within a critical window of size $\sqrt{\ln N} \cdot N^{-1/3}$.

A central quantity in spin glass theory is the overlap between two configurations $\mvgs^1, \mvgs^2$ chosen independently from the Gibbs measure $G_{N,\beta}$, defined as
\begin{align*}
R(\mvgs^1, \mvgs^2) := \frac{1}{N} \sum_{i \leq N} \sigma_i^1 \sigma_i^2
\end{align*}
In the standard SK model with Ising spins, the analysis is significantly more involved due to the discrete nature of the configuration space $\Sigma_N = \{-1, 1\}^N$. The control of fluctuations in this case relies heavily on the concentration of the overlap $R(\mvgs^1, \mvgs^2)$, for which Talagrand~\cite{Tal2} established crucial bounds that become increasingly difficult to control as $\beta \uparrow 1$. Our work builds upon the variance representation via Gaussian interpolation introduced by Chatterjee~\cite{Cha09} and employs Stein's method to provide a refined CLT that remains valid as $\beta_N$ approaches $1$ at the scaling rate $N^{-1/3}$.

\subsection{Main Results}
Our main objective is to characterize the fluctuation of the free energy near, but above the critical temperature $\beta_c = 1$. 

\begin{thm}\label{thm:main}
    Assume that the inverse temperature sequence $(\beta_N)_{N \ge 1}$ satisfies $\lim_{N\to\infty}\beta_N=1$ and 
    \begin{align}\label{eq:beta-condition}
        \lim_{N \to \infty} N^{1/3} (1 - \beta_N^2) = c \in (0,\infty)\cup\{\infty\}.
    \end{align} 
    Define the function
    \begin{align*}
    \nu(\beta) := -\frac{1}{2} \ln (1-\beta^2) - \frac{\beta^2}{2}
    \end{align*}
    for $\beta<1$.
    Then, the variance satisfies
    \begin{align} \label{eq:thm var}
        \abs{\var(F_N(\beta_N)) - \nu(\beta_N)} \leq L \left( N^{1/3}(1-\beta_N^2) \right)^{-3/2}. 
    \end{align}
   Moreover, we have the following distributional convergence for the free energy
    \begin{align} \label{eq:thm dist}
        \dtv\left(\frac{F_N(\beta_N) - \E F_N(\beta_N)}{\sqrt{\nu (\beta_N)}} ,g\right)\le \frac{L}{\sqrt{\nu (\beta_N)}} \left( N^{1/3}(1-\beta_N^2) \right)^{-3/4}.
    \end{align}
Here, \(L\) is a finite constant depending on the sequence \((\beta_N)_{N \geq 1}\), $g\sim\N(0,1)$ and $\dtv$ is the total variation distance. 
\end{thm}

\begin{cor}
In particular, when $\lim_{N \to \infty} N^{1/3} (1 - \beta_N^2) = c \in (0,\infty)$, we have  $\var(F_N(\beta_N))  = \frac16\log N + \cO(1)$ as $N\to\infty$.
\end{cor}

\begin{remark}
    In the physics literature, it is conjectured that the variance of the free energy at the critical temperature $\beta_c = 1$ satisfies
    \begin{align*}
        \var ( F_N(\beta_c)) = \frac{1}{6} \ln N + \cO(1).
    \end{align*}
    To extend the methods developed in this paper to the critical case, one would require precise estimates of the overlap moments at $\beta_c$, similar to Lemma~\ref{lem:R^k_estimate_tal}. For instance, it is conjectured in~\cite[Conjecture 11.7.5]{Tal2}  that 
    \begin{align*}
        \lim_{N \to \infty} N^{2/3} \la R^2(\mvgs^1, \mvgs^2) \ra = a
    \end{align*}
    for some constant $a>0$. Alternatively, one could attempt to compare the free energy at the scaling $\beta_N = 1 - c N^{-1/3}$ with that at $\beta_c = 1$ using a martingale approach. Define the process
    \begin{align*}
        M_N(t) := \sum_{\mvgs \in \Sigma_N} \exp \left( \frac{1}{\sqrt{N}} \sum_{1 \le i<j \le N} B_{ij}(t) \gs_i \gs_j - \frac{(N-1)t}{4} \right)
    \end{align*}
    where $((B_{ij}(t)))_{i<j}$ are independent standard Brownian motions. By Itô's formula, we have
    \begin{align*}
        \ln M_N(\beta_c^2) - \ln M_N(\beta_N^2) = \int_{\beta_N^2}^{\beta_c^2} \frac{1}{M_N(t)} dM_N(t) - \frac{1}{2} \int_{\beta_N^2}^{\beta_c^2} \frac{1}{M_N^2(t)} d\la M_N \ra_t. 
    \end{align*}
    If one can show that the term
    \begin{align*}
         \E \left[ \int_{\beta_N^2}^{\beta_c^2} \frac{1}{M_N^2(t)} d\la M_N \ra_t \right]
         = \int_{\beta_N}^{\beta_c} \beta (N \E \la R^2(\mvgs^1, \mvgs^2) \ra -1 ) d\beta
    \end{align*}  
    is uniformly bounded in $N$, a CLT at the critical temperature could potentially be established. The aforementioned conjecture on the overlap scaling would provide a path to such a proof.
\end{remark}

\begin{remark}
    Our proof approach remains valid in both the high-temperature regime and the near-critical window. Specifically, we extend the classical results of~\cite{ALR87} and \cite{CN95} to sequences of inverse temperatures $(\beta_N)$ that approach the critical value $\beta_c = 1$ at the rate of $N^{-1/3}$.
    By employing a sharper analysis, the variance bound can be improved further. Specifically, for any $\epsilon > 0$, it can be shown that
    \begin{align*} 
        \abs{\var(F_N(\beta_N)) - \nu(\beta_N)} \leq L \left( N^{1/3}(1-\beta_N^2) \right)^{-2 + \epsilon}.
    \end{align*}
\end{remark}

For any real-valued function \(f\) on \(\Sigma_N^n\), we define its Gibbs average under the Gibbs measure \(G_{N,\beta}\) as
\begin{align*}
\la f \ra = \sum_{\mvgs^1, \ldots, \mvgs^n \in \Sigma_N} f(\mvgs^1, \ldots, \mvgs^n) G_{N,\beta}(\mvgs^1) \cdots G_{N,\beta}(\mvgs^n).
\end{align*}

The main technical tools used in our proof are the Gaussian interpolation method and Stein's method for normal approximation.  
In the following two sections, we briefly discuss these now-classical tools. Throughout the rest of the article, we write $H_N$, $Z_N, G_{N}$, and $F_N$ instead of $H_{N,\beta}$, $Z_N(\beta), G_{N,\beta}$, and $F_N(\beta)$, unless we need to explicitly indicate the dependence on $\beta$.  Without loss of generality, we will also assume that $N^{1/3}(1-\beta_N^2)\ge \eps_{0}$ for all $N\ge 1$ and some $\eps_{0}>0$.

\subsection{Coupled Hamiltonians and variance representation}\label{sec:cham}
To analyze the fluctuations, we employ a Gaussian interpolation method.
Let \(y\) be a centered Gaussian vector in \(\dR^n\) with covariance matrix \(C = ((C_{j,j'}))_{1 \le j,j' \le n}\). Consider two independent copies \(y', y''\) of \(y\). For \(t \in [0,1]\), we define the interpolated vectors
\begin{align*}
    y_1 (t) = \sqrt{t}y + \sqrt{1-t} y' \quad \text{and} \quad y_2 (t) = \sqrt{t}y + \sqrt{1-t} y''.
\end{align*}
Let \(A, B: \dR^n \to \dR\) be absolutely continuous functions such that
\begin{align*}
    \E \|\nabla A(y)\|_2^2 < \infty \quad \text{and} \quad \E \|\nabla B(y)\|_2^2 < \infty.
\end{align*}
Using Gaussian integration by parts, the derivative of the coupled expectation can be written as 
\begin{align*}
    \frac{\partial}{\partial t} \E [A(y_1 (t)) B(y_2 (t))] = \sum_{j, j'=1}^n C_{j,j'} \E [\partial_{j} A(y_1 (t)) \partial_{j'} B(y_2 (t))].
\end{align*}
Applying the Fundamental Theorem of Calculus and noting that at $t=1$ the copies coincide (both equal to $y$) while at $t=0$ they are independent (namely $y'$ and $y''$), we derive the covariance formula
\begin{align}
    \E [A(y) B(y)] - \E A(y)\cdot  \E B(y)
    = \int_0^1 \sum_{j, j'=1}^n C_{j,j'} \E [\partial_{j} A(y_1 (t)) \partial_{j'} B(y_2 (t))] \, dt. 
    \label{eq:coupled_GIBP}
\end{align}
Specifically, when $A=B$, the matrix $C$ is positive semi-definite, so the integrand on the right-hand side of~\eqref{eq:coupled_GIBP} is non-negative. This shows the monotonicity of correlations along the interpolation.

We apply this to the SK model. Define two independent copies of the Hamiltonian $H_N^1, H_N^2$ using independent disorder $g', g''$. For $t \in [0,1]$, define the interpolated Hamiltonians
\begin{align*}
    H_{N,t}^1 (\mvgs) &:= \sqrt{t}\cdot H_N (\mvgs) + \sqrt{1-t} \cdot H_N^1 (\mvgs) \text{ and }
    H_{N,t}^2 (\mvgt) := \sqrt{t}\cdot H_N (\mvgt) + \sqrt{1-t}\cdot H_N^2 (\mvgt)
\end{align*}
for $\mvgs,\mvgt\in\gS_N$.
Let $\la \cdot \ra_t$ denote the Gibbs average with respect to the product measure $\tilde{G}_t(\mvgs, \mvgt) \propto \exp(H_{N,t}^1(\mvgs) + H_{N,t}^2(\mvgt))$.
Define the function
\begin{align*}
\xi(x) := \xi_N (x) := \frac{1}{2} \beta^2\cdot \left(x^2 - 1/N\right). 
\end{align*}
A direct application of~\eqref{eq:coupled_GIBP} to the function $A(g) = B(g) = F_N(\beta)$ yields the following variance representation.

\begin{lem}[{\cite[Theorem 3.8]{Cha09}}] \label{lem:var-rep}
    The variance of the free energy is given by
    \begin{align*}
        \var(F_N(\beta)) = N \int_0^1 \E \, \la \xi(R(\mvgs, \mvgt)) \ra_t \, dt.
    \end{align*}
\end{lem}
The proof of Theorem~\ref{thm:main} reduces to establishing the concentration of $N \la R(\mvgs,\mvgt)^2 \ra_t$ near $1/(1-\beta_N^2t)$ and will be proved in Proposition~\ref{prop:R^2 main estimate}. The following lemma will play an important role in the proof, controlling the overlap near the critical temperature.
\begin{lem}\label{lem:R^k_estimate_tal}
    \textnormal{(See \cite[Theorems~11.7.1 and~11.7.2]{Tal2}.)}
    Suppose that (\ref{eq:beta-condition}) holds. 
    Then, for each \(k \geq 1\) we have
    \begin{align*}
    \sup_N \E \la ( N (1-\beta_N^2)R(\mvgs^1, \mvgs^2)^2 )^k\ra < \infty.
    \end{align*}
\end{lem}
However, Lemma~\ref{lem:R^k_estimate_tal} cannot be directly used to control moments of $N(1-\beta_N^{2}t)\la R(\mvgs, \mvgt)^2\ra_{t}, t\in(0,1)$ for the coupled Hamiltonian. We overcome this issue by a careful comparison between the moments under $\la\cdot\ra_{t}$ and $\la \cdot\ra_{1}$, as presented in equation~\eqref{eq:R^2_estimate1} and~\eqref{eq:CLT_est_1}.

Next, we discuss the version of Stein's method for normal approximation that will be useful for us. 

\subsection{Stein's method for normal approximation}
To prove the Central Limit Theorem (CLT), we use Stein's method.
Recall that the total variation distance between two random variables \(X\) and \(Y\) is given by
\begin{align*}
\dtv(X,Y) := \sup_A \abs{ \pr ( X \in A) - \pr ( Y \in A ) }. 
\end{align*}
Stein's method for normal approximation states that the total variation distance between a random variable $W$ and a standard Gaussian random variable $g$ is bounded by the Stein discrepancy
\begin{align}\label{eq:steins_method}
    \dtv(W, g) \leq \sup \left\{ \abs{\E (W \psi(W) - \psi' (W))} : \norm{\psi'}_\infty \leq 2\right\}.
\end{align}

Let us define
\begin{align*}W_N := \frac{F_N - \E F_N}{\sqrt{\nu_N}},\end{align*}
where \(\nu_N := \nu(\beta_N)\).
Using the Gaussian integration by parts formula~\eqref{eq:coupled_GIBP} with $A = W_N$ and $B = \psi(W_N)$, we obtain
\begin{align*}
    \E [W_N \psi (W_N)]
    = \frac{N}{\nu_N} \int_0^1 \E \left[ \psi'(W_{N,t}^1) \, \la \xi(R(\mvgs, \mvgt)) \ra_t \right] dt,
\end{align*}
where \(W^1_{N,t}\) is defined by replacing \(H_N\) in \(W_N\) with \(H^1_{N,t}\).
By evaluating the integral 
\[
\nu_N = \frac{\beta_N^2}{2} \int_0^1 (1/(1-\beta_N^2 t) - 1) dt 
\]
and noting that $W_{N,t}^1\equald W_N$, we obtain the following bound 
\begin{align*}
    \abs{\E [W_N \psi (W_N)] - \E \psi'(W_N)}
    &= \frac{\beta_N^2}{2\nu_N} \abs{ \int_0^1  \E \left[ \psi'(W^1_{N,t}) \left( N\la R(\mvgs, \mvgt)^2 \ra_t - \frac{1}{1-\beta_N^2 t}\right) \right] dt } \nonumber \\
    &\leq \frac{\beta_N^2 \norm{\psi'}_{\infty}}{2\nu_N} \int_0^1  \E \abs{ N\la R(\mvgs, \mvgt)^2 \ra_t - \frac{1}{1-\beta_N^2 t}} dt.
\end{align*}
In particular, we have
\begin{align}\label{eq:stein estimate}
\dtv(W_{N},g) \leq \frac{\beta_N^2}{\nu_N} \int_0^1  \E \abs{ N\la R(\mvgs, \mvgt)^2 \ra_t - \frac{1}{1-\beta_N^2 t}} dt.
\end{align}
Finally, establishing the concentration of $N \la R(\mvgs,\mvgt)^2 \ra_t$ near $1/(1-\beta_N^2t)$ completes the proof of the CLT. The concentration result is presented in Proposition~\ref{prop:R^2 main estimate}.

\subsection{Cavity method}
To estimate the overlap moments, we employ the cavity method to decouple the $N$th spin.
We define the cavity Hamiltonian on $(N-1)$ spins as 
\begin{align*} 
H_{N}^{-} (\mvgs) := \frac{\beta}{\sqrt{N}} \sum_{1 \leq i< j < N} g_{ij} \gs_i \gs_j,
\end{align*}
and define $H_{N}^{1,-}$ and $H_{N}^{2,-}$ analogously by replacing $g_{ij}$ with $g_{ij}'$ and $g_{ij}''$ respectively.
We introduce a second interpolation parameter $s \in [0,1]$ that scales the interaction between the first $(N-1)$ spins and the $N$th spin. For $\mvgr = (\mvgs, \mvgt) \in \Sigma_{N}^2$, define
\begin{align*}
    H_{N,t,s}^1 (\mvgs) 
    := \sqrt{t} H_{N}^- (\mvgs) + \sqrt{1-t} H_{N}^{1,-} (\mvgs) + \sqrt{s} \cdot \frac{\beta}{\sqrt{N}} \sigma_N \sum_{i<N} (\sqrt{t} g_{iN} + \sqrt{1-t}  g_{iN}') \sigma_i,
\end{align*}
\begin{align*}
    H_{N,t,s}^2 (\mvgt) 
    := \sqrt{t} H_{N}^- (\mvgt) + \sqrt{1-t} H_{N}^{2,-} (\mvgt) + \sqrt{s} \cdot \frac{\beta}{\sqrt{N}} \tau_N \sum_{i<N} (\sqrt{t} g_{iN} + \sqrt{1-t}  g_{iN}'') \tau_i.
\end{align*}
Let $\la \cdot \ra_{t,s}$ denote the Gibbs average with respect to the product measure 
$$
\tilde{G}_{t,s}(\mvgs, \mvgt) \propto \exp(H_{N,t,s}^1(\mvgs) + H_{N,t,s}^2(\mvgt)).
$$
By construction, the parameter $s=1$ recovers the full $N$-spin system, while $s=0$ decouples the $N$-th spin. This identification yields the relations 
\(
\la \cdot \ra_{t,1} = \la \cdot \ra_{t}\) and \(\la \cdot \ra_{1,1} = \la \cdot \ra_{1} = \la \cdot \ra.
\)
For \(\mvgs, \mvgt \in \Sigma_N\), define the overlap of the first \((N-1)\) spins as 
\begin{align*}
R^-(\mvgs^1, \mvgs^2) := \frac{1}{N} \sum_{i \leq N-1} \sigma_i^1 \sigma_i^2.
\end{align*}

For \( \mvgr^l = (\mvgs^l, \mvgt^l), \mvgr^{l'} = (\mvgs^{l'}, \mvgt^{l'}) \in \Sigma_{N}^2\), we denote the covariance of the $s$-dependent part of the Hamiltonian $H_{N,t,s}^1(\mvgs) + H_{N,t,s}^2(\mvgt)$. Explicitly, this covariance is defined as
\begin{align*}
    U(\mvgr^l, \mvgr^{l'}) 
    &:= \frac{1}{2} \E \left[ \left( \frac{\beta}{\sqrt{N}} \sigma_N^l \sum_{i<N} (\sqrt{t}g_{iN} + \sqrt{1-t}g_{iN}') \sigma_i^l + \frac{\beta}{\sqrt{N}} \tau_N^l \sum_{i<N} (\sqrt{t}g_{iN} + \sqrt{1-t}g_{iN}'') \tau_i^l \right) \right. \\
    &\qquad \times  \left. \left( \frac{\beta}{\sqrt{N}} \sigma_N^{l'} \sum_{i<N} (\sqrt{t}g_{iN} + \sqrt{1-t}g_{iN}') \sigma_i^{l'} + \frac{\beta}{\sqrt{N}} \tau_N^{l'} \sum_{i<N} (\sqrt{t}g_{iN} + \sqrt{1-t}g_{iN}'') \tau_i^{l'} \right) \right].
\end{align*}
A straightforward computation shows that this expression reduces to
\begin{align*}
    U(\mvgr^l, \mvgr^{l'})
    &= \frac{\beta^2 }{2} ( \sigma_N^l \sigma_N^{l'} R^-(\mvgs^l, \mvgs^{l'}) + \tau_N^l \tau_N^{l'} R^-(\mvgt^l, \mvgt^{l'})) \nonumber \\
    &\quad + \frac{\beta^2 t}{2} ( \sigma_N^l \tau_N^{l'} R^-(\mvgs^l, \mvgt^{l'}) + \tau_N^l \sigma_N^{l'} R^-(\mvgt^l, \mvgs^{l'})).
\end{align*}

Using Gaussian integration by parts with respect to the coupling parameter $s$, we obtain the following useful derivative formula.

\begin{lem}[{\cite[Eq.~(1.90)]{Tal1}}] \label{lem:gauss_ibp}
    For a function \(f\) on \((\Sigma_{N}^2)^n\) and \(t,s \in [0,1]\), we have
    \begin{align*}
        \frac{\partial}{\partial s} \E \la f \ra_{t,s} 
        &= \sum_{l, l' \leq n} \E \la U(\mvgr^l, \mvgr^{l'})f \ra_{t,s}
        - 2n \sum_{l \leq n} \E \la U(\mvgr^l, \mvgr^{n+1})f \ra_{t,s} \\
        &\quad - n \E \la U(\mvgr^{n+1}, \mvgr^{n+1})f \ra_{t,s}
         + n(n+1) \E \la U(\mvgr^{n+1}, \mvgr^{n+2})f \ra_{t,s}.
    \end{align*}
\end{lem}

We proceed by Taylor expansion in $s$ around $s=0$. For \(\mvgr = (\mvgs, \mvgt) \in \Sigma_{N}^2\), consider the function 
\(f_1 (\mvgr) = \sigma_N \tau_N R^- (\mvgs, \mvgt) \).
Since $\sigma_N$ and $\tau_N$ are independent and uniformly distributed on \(\{-1,1\}\) under the expectation \(\E \la \cdot \ra_{t,0}\), all odd moments vanish by symmetry. In particular, we have
\begin{align*}
    \E \la f_1 \ra_{t,0} = 0.
\end{align*}
Calculating derivatives using Lemma~\ref{lem:gauss_ibp}, we find
\begin{align*}
    \left. \frac{\partial}{\partial s} \E \la f_1 \ra_{t,s} \right|_{s=0} 
    &= \beta^2 t \E\la R^-(\mvgs, \mvgt)^2 \ra_{t,0} \text{ and }\\
    \left. \frac{\partial^2}{\partial s^2} \E \la f_1 \ra_{t,s} \right|_{s=0} 
    &= - 4 \beta^4 t \E\la R^-(\mvgs^1, \mvgt^1) R^-(\mvgs^2, \mvgt^1) R^-(\mvgs^1, \mvgs^2) \ra_{t,0}.
\end{align*}
Using Taylor's theorem, we derive the identity
\begin{align}
    \E \la R(\mvgs, \mvgt)^2\ra_{t,1} 
    &= \frac{1}{N} + \E \la f_1 \ra_{t,1} \nonumber \\
    &= \frac{1}{N} + \beta^2 t \E\la R^-(\mvgs, \mvgt)^2 \ra_{t,0} - 2 \beta^4 t \E\la R^-(\mvgs^1, \mvgt^1) R^-(\mvgs^2, \mvgt^1) R^-(\mvgs^1, \mvgs^2) \ra_{t,0} \nonumber \\
    &\qquad + \frac{1}{6} \cdot \left. \frac{\partial^3}{\partial s^3} \E \la f_1 \ra_{t,s} \right|_{s=c}
    \label{eq:main}
\end{align}
for some \(c \in [0,1]\).
The precise control of the error terms yields the following proposition, which is the technical core of this paper.

\begin{prop}\label{prop:R^2 main estimate}
    Under condition~\eqref{eq:beta-condition}, for all \(t \in [0,1]\) and \(N \geq 1\), we have
    \begin{align}
        \textsc{I. } \abs{ N \E \la R(\mvgs, \mvgt)^2\ra_{t} - \frac{1}{1-\beta_N^2 t} }
        &\leq L \cdot N^{-1/2} (1-\beta_N^2)^{-1} (1-\beta_N^2 t)^{-3/2},\label{eq:prop1}\\
         \textsc{II. }  \E \abs{N \la R(\mvgs, \mvgt)^2 \ra_{t} - \frac{1}{1-\beta_N^2 t} }^2
        &\leq L \cdot N^{-3/4} (1-\beta_N^2)^{-53/24}(1-\beta_N^2 t)^{-49/24} \nonumber \\
        &\quad + L \cdot N^{-1/2} (1-\beta_N^2)^{-4/3}(1-\beta_N^2 t)^{-13/6} \label{eq:prop2}
    \end{align}
    where \(L\) is a constant depending on \((\beta_N)_{N\geq 1}\).
\end{prop}

\section{Proofs}
In this section, \(K\) denotes a universal constant, and \(L\) denotes a constant depending on the sequence \((\beta_N)_{N \geq 1}\), in particular it will depend on $\eps_{0}:=\inf_{N}N^{1/3}(1-\beta_{N}^{2})>0$. The values of these constants may change from line to line. We prove Theorem~\ref{thm:main} assuming the validity of Proposition~\ref{prop:R^2 main estimate}. 

\subsection{Proof of Theorem~\ref{thm:main}}
Invoking the variance representation in Lemma~\ref{lem:var-rep}, we obtain
\begin{align*}
    \var(F_N(\beta))
    &= N \int_0^1 \E \, \la \xi(R(\mvgs, \mvgt)) \ra_t \, dt \\
    &= \frac{\beta^2 N}{2} \int_0^1 \E \, \la R(\mvgs, \mvgt)^2 \ra_t \, dt - \frac{\beta^2}{2}.
\end{align*}
By integrating the estimate~\eqref{eq:prop1} from part I.~of Proposition~\ref{prop:R^2 main estimate}, we obtain
\begin{align*}
    \abs{\var(F_N(\beta_N)) - \nu(\beta_N)}
    &= \frac{\beta_N^2}{2} \abs{\int_0^1 \left(N\E \, \la R(\mvgs, \mvgt)^2 \ra_t - \frac{1}{1-\beta_N^2 t}\right) dt} \\
    &\leq L \cdot N^{-1/2} \beta_N^2 (1-\beta_N^2)^{-1} \int_0^{1} (1-\beta_N^2 t)^{-3/2} dt \\
    &= L \cdot \left( N^{1/3}(1-\beta_N^2) \right)^{-3/2}.
\end{align*}
This establishes~\eqref{eq:thm var}. Next, combining~\eqref{eq:stein estimate} and~\eqref{eq:prop2}, we derive the following bound for the Stein discrepancy
\begin{align*}
    &\dtv(W_{N},g) \\
    & \leq L \cdot \frac{\beta_N^2}{\nu_N} \int_0^1  \left[N^{-3/8} (1-\beta_N^2)^{-53/48}(1-\beta_N^2 t)^{-49/48} + N^{-1/4} (1-\beta_N^2)^{-2/3}(1-\beta_N^2 t)^{-13/12}\right]\!dt\\
    & \leq \frac{L}{\nu_N}  \left( N^{1/3} (1-\beta_N^2) \right)^{-3/4}.
\end{align*}
This proves~\eqref{eq:thm dist}.\qed

\subsection{Ingredients for the proof of Proposition~\ref{prop:R^2 main estimate}}

\begin{cor} \label{lem:t,s<t,1}
    For any non-negative function \(f\) on \((\Sigma_{N}^2)^n\),
    \begin{align*}
    \E \la f \ra_{t,s} \leq \exp(2(2n^2 +1) \beta^2(1+t)) \E \la f \ra_{t}.
    \end{align*}
\end{cor}
\begin{proof}
    Recalling that \(\abs{U(\mvgr^l, \mvgr^{l'})} \leq \beta^2 (1+t)\) and applying Lemma~\ref{lem:gauss_ibp}, we have
    \begin{align*}\abs{\frac{\partial}{\partial s} \E \la f \ra_{t,s} }\leq 2(2n^2 +1) \beta^2(1+t) \E \la f \ra_{t,s}.\end{align*}
    The result follows by integrating this differential inequality with respect to \(s\).
\end{proof}

\begin{lem}\label{lem:deriv_estimate2}
    Let \(f\) be a function on \((\Sigma_{N}^2)^n\), and let \(m\ge 1\) be an integer. Then, for all \(\beta, t \in [0,1]\) and \(r_1, r_2 >0\) satisfying \(1/{r_1} + 1/{r_2} = 1\), and such that \(m\cdot r_2\) is an even integer, we have
    \begin{align*}
        \abs{\frac{\partial^m}{\partial s^m} \E \la f \ra_{t,s} } 
        \leq K (\E \la \abs{f}^{r_1} \ra_{t} )^{1/{r_1}} (\E \la \abs{R(\mvgs^1, \mvgs^2)}^{m r_2} \ra )^{1/{r_2}}
    \end{align*}
    where \(K\) is a constant depending on \(n\) and \(m\).
\end{lem}
\begin{proof}
    By Lemma~\ref{lem:gauss_ibp}, \(\abs{\frac{\partial^m}{\partial s^m} \E \la f \ra_{t,s} }\) can be expressed as a linear combination of terms of the form
    \begin{align*}
        \E \la f  \prod_{i=1}^{p} R^-(\mvgs^{l_i}, \mvgs^{l_i'}) \prod_{j=1}^{m-p} R^-(\mvgs^{k_j}, \mvgt^{k_j'})\ra_{t,s}.
    \end{align*}
    Crucially, one can show inductively that \(l_i \neq l_i'\) for all \(i \in [1, p]\) using the fact that 
    \[
    U(\mvgr^l, \mvgr^{l}) = \beta^2 \frac{N-1}{N} + \beta^2 t \sigma_N^l \tau_N^l R^- (\mvgs^l, \mvgt^l).
    \]
    Applying the arithmetic-geometric mean inequality, H\"older's inequality, and Corollary~\ref{lem:t,s<t,1} we obtain
\begin{align}\label{eq:deriv_estimate}
        &\abs{\frac{\partial^m}{\partial s^m} \E \la f \ra_{t,s} }\notag\\
        &\qquad
        \leq K  (\E \la \abs{f}^{r_1} \ra_{t} )^{1/{r_1}}
        \left[ (\E \la \abs{R^-(\mvgs^1, \mvgs^2)}^{m r_2} \ra_{t} )^{1/{r_2}} +  (\E \la \abs{R^-(\mvgs^1, \mvgt^1)}^{m r_2} \ra_{t} )^{1/{r_2}} \right].
    \end{align}
    Expanding the power, we have
    \begin{align*}
        \E \la R^-(\mvgs, \mvgt)^k \ra_{t}
        &= N^{-k} \sum \E \la \sigma_{i_1}\tau_{i_1} \cdots \sigma_{i_k}\tau_{i_k} \ra_{t} \\
        &= N^{-k} \sum \E \left[ \E_1\la \sigma_{i_1} \cdots \sigma_{i_k} \ra_{t} \cdot \E_2\la \tau_{i_1} \cdots \tau_{i_k} \ra_{t}\right]  \\
        &= N^{-k} \sum \E \left( \E_1\la \sigma_{i_1} \cdots \sigma_{i_k} \ra_{t} \right)^2 
    \end{align*}
    where \(\E_1, \E_2\) are expectations with respect to \(g_{ij}', g_{ij}''\) respectively, and the sum runs over all indices of \(i_1, \ldots,i_k \leq N-1\). 
    
    In parallel, \(\E \la R(\mvgs, \mvgt)^k \ra_{t}\) is given by a sum over all indices \(i_1, \ldots,i_k \leq N\). 
    Since the terms are non-negative, we obtain
    \begin{align*}
        \E \la R^-(\mvgs, \mvgt)^k \ra_{t} \leq  \E \la R(\mvgs, \mvgt)^k \ra_{t}
    \end{align*}
    Using the monotonicity argument from the discussion of~\eqref{eq:coupled_GIBP}, we know that \(\E \la R(\mvgs, \mvgt)^k \ra_{t}\) is non-negative and non-decreasing in \(t\). It follows that
    \begin{align}\label{eq:R-(s,t)<R(s,s)}
        \E \la R^-(\mvgs, \mvgt)^k \ra_{t}
        \leq \E \la R(\mvgs, \mvgt)^k \ra_{t}
        \leq \E \la R(\mvgs, \mvgt)^k \ra_{1}
        = \E \la R(\mvgs^1, \mvgs^2)^k \ra.
    \end{align}
    Analogously, we have 
    \begin{align}\label{eq:R-(s,s)<R(s,s)}
        \E \la R^-(\mvgs^1, \mvgs^2)^k \ra_{t}
        = \E \la R^-(\mvgs^1, \mvgs^2)^k \ra
        \leq \E \la R(\mvgs^1, \mvgs^2)^k \ra.
    \end{align}
    Combining~\eqref{eq:deriv_estimate},~\eqref{eq:R-(s,t)<R(s,s)} and~\eqref{eq:R-(s,s)<R(s,s)} completes the proof.
\end{proof}

\subsection{Proof of Proposition~\ref{prop:R^2 main estimate} part I}
In this proof, we work with inverse temperature \(\beta=\beta_N\). Recalling~\eqref{eq:main}, we write
\begin{align}
    (1-
    &\beta_N^2 t) \E \la R(\mvgs, \mvgt)^2\ra_{t} - \frac{1}{N} \nonumber \\
    &=  \beta_N^2 t \left[ \E\la R^-(\mvgs, \mvgt)^2 \ra_{t,0} -\E\la R^-(\mvgs, \mvgt)^2 \ra_{t} \right] + \beta_N^2 t \left[ \E\la R^-(\mvgs, \mvgt)^2 \ra_{t} -\E\la R(\mvgs, \mvgt)^2 \ra_{t} \right] \nonumber \\
    &\quad - 2 \beta_N^4 t \E\la R^-(\mvgs^1, \mvgt^1) R^-(\mvgs^2, \mvgt^1) R^-(\mvgs^1, \mvgs^2) \ra_{t,0} 
    + \frac{1}{6} \cdot \left. \frac{\partial^3}{\partial s^3} \E \la f_1 \ra_{t,s} \right|_{s=c} \label{eq:err-total}
\end{align}
for some \(c \in [0,1]\).
We now proceed to estimate the terms on the right-hand side.

\begin{lem}\label{lem:1st_deriv}
    For all \(N \geq 1\) and \(t \in [0,1]\), there is a constant \(K>0\) such that
    \begin{align*}
        \abs{\E \la R^-(\mvgs, \mvgt)^2\ra_t - \E \la R^-(\mvgs, \mvgt)^2\ra_{t,0}} \leq K\cdot
        \E \la R(\mvgs^1, \mvgs^2)^4\ra.
            \end{align*}
\end{lem}
\begin{proof}
    By Taylor expansion, there exists \(c \in [0,1]\) such that
    \begin{align*}
        \E \la R^-(\mvgs, \mvgt)^2\ra_t 
        = \E \la R^-(\mvgs, \mvgt)^2\ra_{t,0}
        + \left. \frac{\partial}{\partial s} \E \la R^-(\mvgs, \mvgt)^2\ra_{t,s} \right|_{s=0} + \frac{1}{2} \cdot \left. \frac{\partial^2}{\partial s^2} \E \la R^-(\mvgs, \mvgt)^2\ra_{t,s} \right|_{s=c}.
    \end{align*}
    Applying Lemma~\ref{lem:gauss_ibp}, we find that the first derivative vanishes at \(s=0\). Furthermore, by Lemma~\ref{lem:deriv_estimate2}, there exists a constant \(K > 0\) such that
    \begin{align*}
        \abs{ \frac{\partial^2}{\partial s^2} \E \la R^-(\mvgs, \mvgt)^2\ra_{t,s} }
        \leq K\cdot \left( \E \la R^-(\mvgs, \mvgt)^4\ra_{t} \right)^{1/2}
        \left( \E \la R(\mvgs^1, \mvgs^2)^4\ra \right)^{1/2}.
    \end{align*}
    The proof is completed by applying~\eqref{eq:R-(s,t)<R(s,s)}.
\end{proof}

\begin{lem}\label{lem:1st_deriv_2}
For all \(N \geq 1\) and \(t \in [0,1]\), we have
    \begin{align*}
        0 \leq \E\la R(\mvgs, \mvgt)^2 \ra_{t} -\E\la R^-(\mvgs, \mvgt)^2 \ra_{t} 
        \leq \frac{2}{N} \E\la R(\mvgs^1, \mvgs^2)^2 \ra.
    \end{align*}
\end{lem}
\begin{proof}
    First, observe that
    \begin{align*}
        \E\la R(\mvgs, \mvgt)^2 \ra_{t}
        = \frac{1}{N^2} \sum_{i,j \leq N} \E \la \sigma_i \sigma_j \tau_i \tau_j \ra_{t}
        = \frac{1}{N} + \frac{N-1}{N} \E \la \sigma_1 \sigma_2 \tau_1 \tau_2 \ra_{t}.
    \end{align*}
    Similarly, we have
    \begin{align}
        \E\la R^-(\mvgs, \mvgt)^2 \ra_{t}
        &= \frac{N-1}{N^2} + \frac{(N-1)(N-2)}{N^2} \E \la \sigma_1 \sigma_2 \tau_1 \tau_2 \ra_{t} \notag\\
        &= \frac{1}{N^2}+\frac{N-2}{N} \E\la R(\mvgs, \mvgt)^2 \ra_{t} \ge \frac1{2N}.\label{eq:R+}
    \end{align}
    Rearranging these terms and using (\ref{eq:R-(s,t)<R(s,s)}), we obtain
    \begin{align*}
        \E\la R(\mvgs, \mvgt)^2 \ra_{t} - \E\la R^-(\mvgs, \mvgt)^2 \ra_{t}
        = \frac{2}{N} \E\la R(\mvgs, \mvgt)^2 \ra_{t} - \frac{1}{N^2} 
        \leq \frac{2}{N} \E\la R(\mvgs^1, \mvgs^2)^2 \ra.
    \end{align*}
    The left inequality follows from~\eqref{eq:R-(s,t)<R(s,s)} and the non-negativity of \(\E\la R(\mvgs,\mvgt)^2\ra_t\).
\end{proof}

\begin{lem}\label{lem:2nd_deriv}
    For all \(N \geq 1\) and \(t\in [0,1]\), there is a constant \(K>0\) such that
    \begin{align*}
        0 \le \E\la R^-(\mvgs^1, \mvgt^1) R^-(\mvgs^2, \mvgt^1) R^-(\mvgs^1, \mvgs^2) \ra_{t,0} 
        \le K \cdot (\E \la R(\mvgs, \mvgt)^{2}\ra_{t})^{1/2} (\E \la R(\mvgs^1, \mvgs^2)^4\ra)^{1/2}.
    \end{align*}
\end{lem}
\begin{proof}
    Note that
    \begin{align*}
        \E\la R^-(\mvgs^1, \mvgt^1) 
        &R^-(\mvgs^2, \mvgt^1) R^-(\mvgs^1, \mvgs^2) \ra_{t,0}\\
        &= \frac{1}{N} \sum_{i<N} \E \la \sigma_i^1 \sigma_i^2 R^-(\mvgs^1, \mvgt^1) R^-(\mvgs^2, \mvgt^1) \ra_{t,0} \\
        &= \frac{1}{N} \sum_{i<N} \E  \sum_{\mvgt^1 \in \Sigma_N} \left( \sum_{\mvgs^1 \in \Sigma_N}\sigma_i^1 R^-(\mvgs^1, \mvgt^1) G_{t,0}^1(\mvgs^1) \right)^2 G_{t,0}^2(\mvgt^1)
        \geq 0.
    \end{align*}
    For the upper bound, we apply H\"{o}lder's inequality and Corollary~\ref{lem:t,s<t,1}
    \begin{align*}
        \E\la R^-(\mvgs^1, \mvgt^1)
        &R^-(\mvgs^2, \mvgt^1) R^-(\mvgs^1, \mvgs^2) \ra_{t,0} \\
        & \leq (\E \la R^-(\mvgs, \mvgt)^{2}\ra_{t,0})^{1/2} (\E \la R^-(\mvgs, \mvgt)^{4}\ra_{t,0})^{1/4} (\E \la R^-(\mvgs^1, \mvgs^2)^4\ra_{t,0})^{1/4} \\
        & \leq K (\E \la R^-(\mvgs, \mvgt)^{2}\ra_{t})^{1/2} (\E \la R^-(\mvgs, \mvgt)^{4}\ra_{t})^{1/4} (\E \la R^-(\mvgs^1, \mvgs^2)^4\ra_{t})^{1/4}.
    \end{align*}
    Inequalities~\eqref{eq:R-(s,t)<R(s,s)} and~\eqref{eq:R-(s,s)<R(s,s)} complete the proof.
\end{proof}

\begin{lem}\label{lem:3rd_deriv}
    For \(\mvgr = (\mvgs,\mvgt) \in \Sigma_{N}^2\), let \(f_1 = \sigma_N \tau_N R^- (\mvgs, \mvgt) \).
    For all \(N \geq 1\) and \(t,s \in [0,1]\), there is a constant \(K > 0\) such that
    \begin{align*}
        \abs{\frac{\partial^3}{\partial s^3} \E \la f_1 \ra_{t,s}} 
        \leq K \cdot \E \la R (\mvgs^1, \mvgs^2)^{4} \ra .
    \end{align*} 
\end{lem}
\begin{proof}
    Using Lemma~\ref{lem:deriv_estimate2} with $m=2, r_{1}=r_{2}=2$ and~\eqref{eq:R-(s,t)<R(s,s)}, we deduce
    \begin{align*}
        \abs{\frac{\partial^3}{\partial s^3} \E \la f_1 \ra_{t,s}}
        &\leq K  \cdot\left( \E \la R^- (\mvgs, \mvgt)^{4} \ra_{t} \right)^{1/2} \left( \E \la R (\mvgs^1, \mvgs^2)^{4} \ra \right)^{1/2} 
        \leq K \cdot  \E \la R (\mvgs^1, \mvgs^2)^{4} \ra.
    \end{align*}
    This completes the proof.
\end{proof}

We cannot use Lemma~\ref{lem:R^k_estimate_tal} to control $\E \la ( N (1-\beta_N^2 t)R(\mvgs, \mvgt)^2 )^k\ra_{t}$ for general $k$. However,  we can bound it for $k=2$ as follows.
Substituting the estimates from Lemmas~\ref{lem:1st_deriv},~\ref{lem:1st_deriv_2},~\ref{lem:2nd_deriv} and~\ref{lem:3rd_deriv} into~\eqref{eq:err-total}, we obtain
\begin{align*}
    (1-\beta_{N}^2 t) \E \la R(\mvgs, \mvgt)^2\ra_{t} - \frac{1}{N} 
    \leq K \cdot \E \la R(\mvgs^1, \mvgs^2)^4\ra.
\end{align*}
Applying Lemma~\ref{lem:R^k_estimate_tal} and condition~\eqref{eq:beta-condition},
\begin{align}\label{eq:R^2_estimate1}
    N(1-\beta_{N}^2 t) \E \la R(\mvgs, \mvgt)^2\ra_{t} 
    \leq 1+L \cdot N^{-1/3} (N^{1/3}(1- \beta_N^2))^{-2} =\cO(1).
\end{align}

Similarly, using~\eqref{eq:R+} and \eqref{eq:R-(s,t)<R(s,s)}, we find
\begin{align*}
    \abs{ (1-\beta_N^2 t) \E \la R(\mvgs, \mvgt)^2\ra_{t} - \frac{1}{N} }
    &\leq K \cdot \E \la R(\mvgs^1, \mvgs^2)^4\ra + \frac{2}{N} \E\la R(\mvgs^1, \mvgs^2)^2 \ra\\
    &\qquad + K \cdot (\E \la R(\mvgs, \mvgt)^{2}\ra_{t})^{1/2} (\E \la R(\mvgs^1, \mvgs^2)^4\ra)^{1/2}\\
    &\leq  K \cdot (\E \la R(\mvgs, \mvgt)^{2}\ra_{t})^{1/2} (\E \la R(\mvgs^1, \mvgs^2)^4\ra)^{1/2}.
\end{align*}
Combining the above with Lemma~\ref{lem:R^k_estimate_tal} and~\eqref{eq:R^2_estimate1}, we obtain
\begin{align*}
    \abs{ (1-\beta_N^2 t) \E \la R(\mvgs, \mvgt)^2\ra_{t} - \frac{1}{N} }
    \leq L \cdot N^{-3/2} (1-\beta_N^2)^{-1} (1-\beta_N^2 t)^{-1/2}.
\end{align*}
Here, we used that $1-\beta_{N}^{2}t\le 1 \ll N(1-\beta_{N}^{2})^{2}$.
This completes the proof of~\eqref{eq:prop1}.\qed

\subsection{Proof of Proposition~\ref{prop:R^2 main estimate} part II}
First, we have the identity
\begin{align}\label{eq:varbd}
\begin{split}
    &\E \left( N \la R(\mvgs, \mvgt)^2 \ra_{t} - \frac{1}{1-\beta_N^2 t} \right)^2\\
    &\qquad= \frac{N^2}{1-\beta_N^2 t}\left[ (1-\beta_N^2 t)\E \la R(\mvgs^1, \mvgt^1)^2 R(\mvgs^2, \mvgt^2)^2 \ra_{t}  - \frac{1}{N} \E \la R(\mvgs, \mvgt)^2 \ra_{t} \right] \\
    &\qquad\qquad - \frac{1}{1-\beta^2_N t} \left[ N\E \la R(\mvgs, \mvgt)^2 \ra_{t} - \frac{1}{1-\beta_N^2 t} \right].
    \end{split}
\end{align}
By  Part I.~of Proposition~\ref{prop:R^2 main estimate} we can bound the second term in~\eqref{eq:varbd} by
\begin{align*}
\frac{1}{1-\beta^2_N t} \abs{ N \E \la R(\mvgs, \mvgt)^2\ra_{t} - \frac{1}{1-\beta_N^2 t} }
        &\leq L \cdot N^{-1/2} (1-\beta_N^2)^{-1} (1-\beta_N^2 t)^{-5/2}.
\end{align*}
Thus, we need to bound
\begin{align}\label{eq:var2bd}
(1-\beta_N^2 t)\E \la R(\mvgs^1, \mvgt^1)^2 R(\mvgs^2, \mvgt^2)^2 \ra_{t}  - \frac{1}{N} \E \la R(\mvgs, \mvgt)^2 \ra_{t} .
\end{align}
For \(\mvgr^1=(\mvgs^{1},\mvgt^{1}), \mvgr^2=(\mvgs^{2},\mvgt^{2}) \in \Sigma_{N}^2\), let us define a function
\begin{align*}f_2(\mvgr^1, \mvgr^2) = \sigma_N^1 \tau_N^1 R^-(\mvgs^1, \mvgt^1) R(\mvgs^2, \mvgt^2)^2.\end{align*}
By exchangeability of the spin coordinates, we have the decomposition
\begin{align*}
    \E \la R(\mvgs^1, \mvgt^1)^2 R(\mvgs^2, \mvgt^2)^2 \ra_{t}
    &= \E \la \sigma_N^1 \tau_N^1  R(\mvgs^1, \mvgt^1) R(\mvgs^2, \mvgt^2)^2 \ra_{t} 
    = \E \la f_2 \ra_{t} + \frac{1}{N} \E \la R(\mvgs, \mvgt)^2 \ra_{t}
\end{align*}
Since \(\sigma_N^1\) and \(\tau_N^1\) are uniformly distributed on \(\{-1,1\}\) under \(\E \la \cdot \ra_{t,0}\), it follows that
\begin{align*}
    \E \la f_2 \ra_{t,0} = 0.
\end{align*}
Using Gaussian integration by parts (Lemma~\ref{lem:gauss_ibp}), we compute the derivative:
\begin{align*}
    \left. \frac{\partial}{\partial s} \E \la f_2 \ra_{t,s} \right|_{s=0} 
    = \beta^2 t \E \la R^-(\mvgs^1, \mvgt^1)^2 R(\mvgs^2, \mvgt^2)^2 \ra_{t,0}.
\end{align*}
Applying Taylor's theorem, we expand \(\E \la f_2 \ra_{t,s}\) as follows
\begin{align*}
    \E \la R(\mvgs^1, &\mvgt^1)^2 R(\mvgs^2, \mvgt^2)^2 \ra_{t}\\
    &= \E \la f_2 \ra_{t,0} + \left. \frac{\partial}{\partial s} \E \la f_2 \ra_{t,s} \right|_{s=0} + \frac{1}{2} \left. \frac{\partial^2}{\partial s^2} \E \la f_2 \ra_{t,s} \right|_{s=c} + \frac{1}{N} \E \la R(\mvgs, \mvgt)^2 \ra_{t} \\
    &= \beta^2 t \E \la R^-(\mvgs^1, \mvgt^1)^2 R(\mvgs^2, \mvgt^2)^2 \ra_{t,0}  + \frac{1}{2} \left. \frac{\partial^2}{\partial s^2} \E \la f_2 \ra_{t,s} \right|_{s=c} + \frac{1}{N} \E \la R(\mvgs, \mvgt)^2 \ra_{t}.
\end{align*}
for some \(c \in [0,1]\).
Rearranging the terms yields
\begin{align}
    \eqref{eq:var2bd}&=(1- \beta^2 t) \E \la R(\mvgs^1, \mvgt^1)^2 R(\mvgs^2, \mvgt^2)^2 \ra_{t} - \frac{1}{N} \E \la R(\mvgs, \mvgt)^2 \ra_{t} \nonumber \\
    &= \beta^2 t \left[ \E \la R^-(\mvgs^1, \mvgt^1)^2 R(\mvgs^2, \mvgt^2)^2 \ra_{t,0} - \E \la R^-(\mvgs^1, \mvgt^1)^2 R(\mvgs^2, \mvgt^2)^2 \ra_{t} \right] \label{eq:CLT_main}\\
    & - \beta^2 t \left[ \E \la R(\mvgs^1, \mvgt^1)^2 R(\mvgs^2, \mvgt^2)^2 \ra_{t} -\E \la R^-(\mvgs^1, \mvgt^1)^2 R(\mvgs^2, \mvgt^2)^2 \ra_{t} \right]  + \frac{1}{2} \left. \frac{\partial^2}{\partial s^2} \E \la f_2 \ra_{t,s} \right|_{s=c}.\notag
\end{align}
We now estimate each term on the right-hand side of~\eqref{eq:CLT_main}.

\begin{lem}\label{lem:CLT2}
    For all \(N \geq 1\) and \(t \in [0,1]\), we have
    \begin{align*}
        \E \la R(\mvgs^1, \mvgt^1)^2 R(\mvgs^2, \mvgt^2)^2 \ra_{t} - \E \la R^-(\mvgs^1, \mvgt^1)^2 R(\mvgs^2, \mvgt^2)^2 \ra_{t} 
        \geq 0.
    \end{align*}
\end{lem}
\begin{proof}
    We observe that
    \begin{align*}
        \E \la R(\mvgs^1, \mvgt^1)^2 R(\mvgs^2, \mvgt^2)^2 \ra_{t}
        = \frac{N-1}{N} \E \la \sigma^1_1 \sigma^1_2 \tau^1_1 \tau^1_2 R(\mvgs^2, \mvgt^2)^2 \ra_{t} + \frac{1}{N} \E \la R(\mvgs, \mvgt)^2 \ra_{t}
    \end{align*}
    and
    \begin{align*}
        \E \la R^-(\mvgs^1, \mvgt^1)^2 R(\mvgs^2, \mvgt^2)^2 \ra_{t}
        = \frac{(N-1)(N-2)}{N^2} \E \la \sigma^1_1 \sigma^1_2 \tau^1_1 \tau^1_2 R(\mvgs^2, \mvgt^2)^2 \ra_{t} + \frac{N-1}{N^2} \E \la R(\mvgs, \mvgt)^2 \ra_{t}.
    \end{align*}
    Furthermore,
    \begin{align*}
        \E \la \sigma^1_1 \sigma^1_2 \tau^1_1 \tau^1_2 R(\mvgs^2, \mvgt^2)^2 \ra_{t}
        &= \frac{1}{N^2} \E \la \sigma_1 \sigma_2 \ra_t \la \tau_1 \tau_2 \ra_t \la \sum_{i,j \le N} \sigma_i \sigma_j \tau_i \tau_j \ra_{t}\\
        &= \frac{1}{N^2} \sum_{i,j \le N} \E \left[ \E_1 \left[ \la \sigma_1 \sigma_2 \ra_t \la \sigma_i \sigma_j \ra_t \right] \E_2 \left[ \la \tau_1 \tau_2 \ra_t  \la \tau_i \tau_j \ra_{t} \right] \right] \\
        &= \frac{1}{N^2} \sum_{i,j \le N} \E \left( \E_1 \left[ \la \sigma_1 \sigma_2 \ra_t \la \sigma_i \sigma_j \ra_t \right]  \right)^2 \geq 0
    \end{align*}
    where \(\E_1\) and \(\E_2\) denote expectations with respect to \((g'_{ij})_{i,j \le N}, (g''_{ij})_{i,j \le N}\) respectively. This completes the proof.
\end{proof}

\begin{lem}\label{lem:CLT1}
    For all \(N \geq 1\) and \(t \in [0,1]\), there is a constant \(K>0\) such that
    \begin{align*}
        &\abs{\E \la R^-(\mvgs^1, \mvgt^1)^2 R(\mvgs^2, \mvgt^2)^2 \ra_{t,0} - \E \la R^-(\mvgs^1, \mvgt^1)^2 R(\mvgs^2, \mvgt^2)^2 \ra_{t}} \\ 
        &\qquad \leq K\cdot \left(\E \la R(\mvgs^1, \mvgt^1)^2 R(\mvgs^2, \mvgt^2)^2  \ra_{t} \right)^{1/2}  (\E \la R(\mvgs^1, \mvgs^2)^{12} \ra )^{1/3}
        \leq K \cdot \left( \E \la R(\mvgs^1, \mvgs^2)^{12} \ra \right)^{1/2}.
    \end{align*}
\end{lem}
\begin{proof}
    By Lemma~\ref{lem:gauss_ibp}, we have
    \begin{align*}
        \left. \frac{\partial}{\partial s} \E \la R^-(\mvgs^1, \mvgt^1)^2 R(\mvgs^2, \mvgt^2)^2 \ra_{t,s} \right|_{s=0} 
        = 0.
    \end{align*}
    Using Lemma~\ref{lem:deriv_estimate2} with $m=2,r_{1}=3/2,r_{2}=3$, H\"older's inequality and~\eqref{eq:R-(s,t)<R(s,s)}, we obtain
        \begin{align*}
         &\abs{ \frac{\partial^2}{\partial s^2} \E \la R^-(\mvgs^1, \mvgt^1)^2 R(\mvgs^2, \mvgt^2)^2 \ra_{t,s} }\\
         &\qquad \leq K \left(\E \la \left( R^-(\mvgs^1, \mvgt^1)^2 R(\mvgs^2, \mvgt^2)^2 \right)^{3/2} \ra_{t} \right)^{2/3} (\E \la R(\mvgs^1, \mvgs^2)^{6} \ra )^{1/3}  \\
         &\qquad \leq K \left(\E \la R^-(\mvgs^1, \mvgt^1)^2 R(\mvgs^2, \mvgt^2)^2  \ra_{t} \right)^{1/2} \left(\E \la R^-(\mvgs^1, \mvgt^1)^6 R(\mvgs^2, \mvgt^2)^6  \ra_{t} \right)^{1/6} (\E \la R(\mvgs^1, \mvgs^2)^{6} \ra )^{1/3}  \\
         &\qquad \leq K \left(\E \la R(\mvgs^1, \mvgt^1)^2 R(\mvgs^2, \mvgt^2)^2  \ra_{t} \right)^{1/2}  (\E \la R(\mvgs^1, \mvgs^2)^{12} \ra )^{1/3} \\
         &\qquad \leq K (\E \la R(\mvgs^1, \mvgs^2)^{12} \ra )^{1/2}.
    \end{align*}
    In the third inequality, we used Lemma~\ref{lem:CLT2}.
    Taylor's theorem completes the proof.
\end{proof}

\begin{lem}\label{lem:CLT3}
    For \(\mvgr^1, \mvgr^2 \in \Sigma_{N}^2\), consider the function \(f_2(\mvgr^1, \mvgr^2) = \sigma_N^1 \tau_N^1 R^-(\mvgs^1, \mvgt^1) R(\mvgs^2, \mvgt^2)^2 \).
    For all \(N \geq 1\) and \(t,s \in [0,1]\), there is a constant \(K > 0\) such that
    \begin{align}\label{eq:lem_CLT3_1}
        \abs{ \frac{\partial^2}{\partial s^2} \E \la f_2 \ra_{t,s} }
        &\leq K (\E \la R(\mvgs, \mvgt)^{2} \ra_t )^{3/4} (\E \la R(\mvgs^1, \mvgs^2)^{16} \ra )^{7/32}\\
        \abs{ \frac{\partial^2}{\partial s^2} \E \la f_2 \ra_{t,s} }
        &\le K (\E \la R(\mvgs^1, \mvgt^1)^{2} R(\mvgs^2, \mvgt^2)^2 \ra_t )^{1/2} (\E \la R(\mvgs, \mvgt)^{2} \ra_t )^{1/6} (\E \la R(\mvgs^1, \mvgs^2)^{8} \ra )^{1/3}.    \label{eq:lem_CLT3_2}
    \end{align}
\end{lem}
\begin{proof}
    Applying Lemma~\ref{lem:deriv_estimate2}, H\"older's inequality and~\eqref{eq:R-(s,t)<R(s,s)}, we have
    \begin{align*}
         \abs{ \frac{\partial^2}{\partial s^2} \E \la f_2 \ra_{t,s} }
         &\leq K (\E \la \abs{f_2}^{8/7} \ra_{t} )^{7/8} (\E \la \abs{R(\mvgs^1, \mvgs^2)}^{16} \ra )^{1/8} \\
         &\leq K (\E \la R(\mvgs^2, \mvgt^2)^{2} \ra_t )^{3/4} (\E \la R^-(\mvgs^1, \mvgt^1)^{8} R(\mvgs^2, \mvgt^2)^{4} \ra_t )^{1/8} (\E \la R(\mvgs^1, \mvgs^2)^{16} \ra )^{1/8} \\
         &\leq K (\E \la R(\mvgs, \mvgt)^{2} \ra_t )^{3/4} (\E \la R(\mvgs^1, \mvgs^2)^{16} \ra )^{7/32}.
    \end{align*}
    Similarly,
    \begin{align*}
         \abs{ \frac{\partial^2}{\partial s^2} \E \la f_2 \ra_{t,s} }
         &\leq K (\E \la \abs{f_2}^{4/3} \ra_{t} )^{3/4} (\E \la \abs{R(\mvgs^1, \mvgs^2)}^{8} \ra )^{1/4} \\
         &\leq K (\E \la R^-(\mvgs^1, \mvgt^1)^{2} R(\mvgs^2, \mvgt^2)^2 \ra_t )^{1/2} (\E \la R(\mvgs, \mvgt)^{4} \ra_t )^{1/4} (\E \la R(\mvgs^1, \mvgs^2)^{8} \ra )^{1/4} \\
         &\leq K (\E \la R(\mvgs^1, \mvgt^1)^{2} R(\mvgs^2, \mvgt^2)^2 \ra_t )^{1/2} (\E \la R(\mvgs, \mvgt)^{2} \ra_t )^{1/6} (\E \la R(\mvgs^1, \mvgs^2)^{8} \ra )^{1/3}.
    \end{align*}
    In the last inequality, we used Lemma~\ref{lem:CLT2}.
\end{proof}

Note that, if we have the bound that $N^2(1-\beta_N^2t)^2\E \la R(\mvgs^1, \mvgt^1)^4 \ra_{t} \leq L $ for all $t\in[0,1]$, the proof of Lemmas~\ref{lem:CLT1} and \ref{lem:CLT3} would have been simpler. However, in the absence of such a bound, we use a comparison result to control the error terms, as given below.

We are now ready to prove~\eqref{eq:prop2} of Proposition~\ref{prop:R^2 main estimate}. Substituting~\eqref{eq:lem_CLT3_1} and the bounds from Lemmas~\ref{lem:CLT2} and~\ref{lem:CLT1} into~\eqref{eq:CLT_main}, we obtain
\begin{align*}
    (1- &\beta_N^2 t) 
    \E \la R(\mvgs^1, \mvgt^1)^2 R(\mvgs^2, \mvgt^2)^2 \ra_{t}\\
    &\leq K \cdot \left( \E \la R(\mvgs^1, \mvgs^2)^{12} \ra \right)^{1/2} + K \cdot (\E \la R(\mvgs, \mvgt)^{2} \ra_t )^{3/4} (\E \la R(\mvgs^1, \mvgs^2)^{16} \ra )^{7/32} + \frac{1}{N} \E \la R(\mvgs, \mvgt)^2 \ra_{t}.
\end{align*}
The first and third terms can be bounded by \(L \cdot N^{-2}(1-\beta_N^2 t)^{-1}\) using Lemma~\ref{lem:R^k_estimate_tal} and~\eqref{eq:R^2_estimate1}. An upper bound for the second term is obtained analogously. Thus, we have 
\begin{align}
    (1- \beta_N^2 t) &
    \E \la R(\mvgs^1, \mvgt^1)^2 R(\mvgs^2, \mvgt^2)^2 \ra_{t} \nonumber \\
    &\leq L \cdot N^{-5/2} (1-\beta_N^2)^{-7/4}(1-\beta_N^2 t)^{-3/4} + L \cdot N^{-2}(1-\beta_N^2 t)^{-1}.\label{eq:CLT_est_1}
\end{align}
Similarly, combining~\eqref{eq:lem_CLT3_2}, Lemmas~\ref{lem:CLT2} and~\ref{lem:CLT1} with~\eqref{eq:CLT_main}, we obtain
\begin{align*}
    (1- \beta_N^2 t) &\E \la R(\mvgs^1, \mvgt^1)^2 R(\mvgs^2, \mvgt^2)^2 \ra_{t} - \frac{1}{N} \E \la R(\mvgs, \mvgt)^2 \ra_{t}  \\
    &\qquad\leq K \left(\E \la R(\mvgs^1, \mvgt^1)^2 R(\mvgs^2, \mvgt^2)^2  \ra_{t} \right)^{1/2}  (\E \la R(\mvgs^1, \mvgs^2)^{12} \ra )^{1/3} \\
    &\qquad\quad +  K (\E \la R(\mvgs^1, \mvgt^1)^{2} R(\mvgs^2, \mvgt^2)^2 \ra_t )^{1/2} (\E \la R(\mvgs, \mvgt)^{2} \ra_t )^{1/6} (\E \la R(\mvgs^1, \mvgs^2)^{8} \ra )^{1/3}.
\end{align*}
Using Lemma~\ref{lem:R^k_estimate_tal} and~\eqref{eq:R^2_estimate1} again, the terms 
$$
(\E \la R(\mvgs^1, \mvgs^2)^{12} \ra )^{1/3}\text{ and }
(\E \la R(\mvgs, \mvgt)^{2} \ra_t )^{1/6} (\E \la R(\mvgs^1, \mvgs^2)^{8} \ra )^{1/3}
$$
are bounded by \(L \cdot N^{-3/2} (1-\beta_N^2 t)^{-1/6} (1-\beta_N^2)^{-4/3}\).
Additionally, we use~\eqref{eq:CLT_est_1} along with the inequality \((a+b)^{1/2} \leq a^{1/2} + b^{1/2}\) to estimate 
\(
\left(\E \la R(\mvgs^1, \mvgt^1)^2 R(\mvgs^2, \mvgt^2)^2  \ra_{t} \right)^{1/2}.
\)
Thus, we obtain another upper bound
\begin{align}
    &(1- \beta_N^2 t) \E \la R(\mvgs^1, \mvgt^1)^2 R(\mvgs^2, \mvgt^2)^2 \ra_{t} - \frac{1}{N} \E \la R(\mvgs, \mvgt)^2 \ra_{t}  \nonumber \\
    &\qquad\leq L \cdot N^{-11/4} (1-\beta_N^2)^{-53/24}(1-\beta_N^2 t)^{-25/24} + L \cdot N^{-5/2} (1-\beta_N^2)^{-4/3}(1-\beta_N^2 t)^{-7/6}. \label{eq:CLT_est_2}
\end{align}

Combining this with~\eqref{eq:prop1} and~\eqref{eq:varbd}, we arrive at the desired bound
\begin{align*}
    \E \left( N \la R(\mvgs, \mvgt)^2 \ra_{t} - \frac{1}{1-\beta_N^2 t} \right)^2
    &\leq L \cdot N^{-3/4} (1-\beta_N^2)^{-53/24}(1-\beta_N^2 t)^{-49/24} \\
    &\quad + L \cdot N^{-1/2} (1-\beta_N^2)^{-4/3}(1-\beta_N^2 t)^{-13/6}.
\end{align*}%%%%%%%%%%%%%%%%%%%%%%%%%%%%%%%

\paragraph{\bfseries Acknowledgements.} The authors would like to thank Qiang Wu for helpful discussions. 

\bibliographystyle{alphaurl}
\bibliography{crit_temp} 
\end{document}